\documentclass[11pt]{amsart}

\usepackage{amssymb,amsmath,accents}
\usepackage{bbm}
\usepackage{a4wide}

\usepackage[colorlinks=true, pdfstartview=FitV, linkcolor=blue, 
citecolor=blue]{hyperref}

\newtheorem{theorem}{Theorem}[section]
\newtheorem{prop}[theorem]{Proposition}
\newtheorem{lemma}[theorem]{Lemma}
\newtheorem{fact}[theorem]{Fact}
\newtheorem{coro}[theorem]{Corollary}
\theoremstyle{definition}
\newtheorem{definition}[theorem]{Definition}
\newtheorem{example}[theorem]{Example}
\newtheorem{remark}[theorem]{Remark}

\newcommand{\ts}{\hspace{0.5pt}}
\newcommand{\nts}{\hspace{-0.5pt}}
\newcommand{\AAA}{\mathbb{A}}
\newcommand{\RR}{\mathbb{R}\ts}
\newcommand{\NN}{\mathbb{N}}
\newcommand{\PP}{\mathbb{P}}
\newcommand{\XX}{\mathbb{X}}
\newcommand{\YY}{\mathbb{Y}}

\newcommand{\cA}{\mathcal{A}}
\newcommand{\cB}{\mathcal{B}}
\newcommand{\cM}{\mathcal{M}}
\newcommand{\cP}{\mathcal{P}}
\newcommand{\cU}{\ts\ts\mathcal{U}}

\newcommand{\ee}{\ts\mathrm{e}}
\newcommand{\nexp}{\mathrm{Exp}}
\newcommand{\nlog}{\mathrm{Log}}
\newcommand{\dd}{\,\mathrm{d}}
\newcommand{\one}{\mathbbm{1}}
\newcommand{\trans}{{\scriptscriptstyle \mathsf{T}}}

\newcommand{\udo}[1]{\underaccent{$\text{.}$}{#1\ts}\nts}
\newcommand{\exend}{\hfill$\Diamond$}
\newcommand{\pa}{\phantom{a}}

\newcommand{\Mat}{\mathrm{Mat}}

\newcommand{\defeq}{\mathrel{\mathop:}=}
\newcommand{\eqdef}{=\mathrel{\mathop:}}

\newcommand{\myfrac}[2]{\frac{\raisebox{-2pt}{$#1$}}
  {\raisebox{0.5pt}{$#2$}}}

\begin{document}

\title[Coupon collection and Markov embedding]
 {A multiple coupon collection process\\[2mm]
  and its Markov embedding structure}

\author{Ellen Baake}
\address{Technische Fakult\"at, Universit\"at Bielefeld, 
         Postfach 100131, 33501 Bielefeld, Germany}

\author{Michael Baake}
\address{Fakult\"at f\"ur Mathematik, Universit\"at Bielefeld, 
         Postfach 100131, 33501 Bielefeld, Germany}

\begin{abstract} 
  The embedding problem of Markov transition matrices into 
   continuous-time Markov semigroups is a classic problem that
  regained a lot of impetus and activities in recent years.  We
  consider it here for the following generalisation of the well-known
  coupon collection process: from a finite set of distinct objects, a
  subset is drawn repeatedly according to some probability
  distribution, independently and with replacement, and each time
  united with the set of objects sampled so far. We derive and
  interpret properties of and explicit conditions for the resulting
  discrete-time Markov chain to be representable within a semigroup or
  a flow of a continuous-time process of the same type.
\end{abstract}

\keywords{Markov chains, coupon collection, embedding problem,
                  incidence algebra}
\subjclass[2010]{60J10, 60J27, 15A16, 06A07}

\maketitle

\section{Introduction}
The classic coupon collector's problem (CCP) starts from a set
$S = \{1,2,\ldots, N \}$ of distinct types of objects (coupons). The
objects are hidden in breakfast cereal boxes --- exactly one in each
box, its type distributed according to $\cU(S)$, the uniform
distribution on $S$, independently for every box. The collector opens
one box at a time, that is, he samples with replacement from $\cU(S)$,
and stops when he has all types in $S$. One is typically interested in
the distribution of the number of boxes opened, that is, the number of
samples required.

The classic CCP has been generalised in various ways.  Neal
\cite{Neal} investigates the situation that the objects may have
different frequencies (and may be absent), so that $\cU(S)$ is
replaced by a discrete distribution $p = (p_i)_{i \in S \cup \{0\}}$
on $S \cup \{0\}$ with $0<p_i < 1$ for $i \in S \cup \{0\}$ and 
  $\sum_{i=0}^N p_i=1$.  Schilling and Henze \cite{SchillingHenze}
  allow for a fixed number $1 \leqslant k < N$ of different objects
per box, uniformly on the set of subsets of $S$ of size $k$, or
slightly non-uniformly.
  
Here, we consider an even more general setting, the \emph{multiple
  coupon collection process} (MCCP), which we define as follows.
Every box contains a random subset of $S$, where subset
$K \subseteq S$ is present with probability $p^{\pa}_K$, with
$\sum_{K \subseteq S} p^{\pa}_K = 1$, again independently for every
box. Our interest here is in the discrete-time Markov chain
$X=(X_n)_{n \in \NN_0}$ on the set of subsets of $S$, where $X_n$ is
the set of types collected until step $n$.  While we do not consider
the usual stopping time problem, we are concerned with whether $X$ is
\emph{embeddable}, that is, whether there is a continuous-time Markov
chain $(Y^{\pa}_t)^{\pa}_{t \in \RR_{\geqslant 0}}$ such that
$\PP(Y^{\pa}_{1}=J \mid Y^{\pa}_{0}=I) = \PP(X^{\pa}_{1}=J \mid X^{\pa}_{0}=I)$
for all $I,J \subseteq S$.  Put differently, we ask the question under
which conditions on the distribution $p = (p^{\pa}_K)^{\pa}_{K \subseteq S}$
the Markov transition matrix of $X$ occurs in a time-homogeneous
Markov semigroup of the form $\{ \ee^{t \ts Q} : t \geqslant 0\}$ for
some Markov generator $Q$. More generally, one can also consider the
time-inhomogeneous situation of a general Markov flow; compare
\cite{Joh, BS4} and references therein. For our special matrix family,
this extension does not seem to provide extra insight, wherefore we
only touch upon it briefly.

Concrete and exhaustive criteria for a discrete-time Markov chain to
be embeddable this way are known for state spaces of up to cardinality
$4$; see \cite{CFR, BS3} and references therein. Beyond this, they are
restricted to specific families of Markov matrices; see \cite{BS2,
  BS4} for examples. The MCCP enriches this collection by an
interesting member for which a complete, yet non-trivial answer can be
given. This is partly due to the fact that the transition matrix is of
triangular form, and partly due to a fruitful combination of
algebraic, combinatorial and probabilistic tools. The MCCP thus adds
relevant insight into the embeddability problem, a topic that has a
long history \cite{Elfving, King, Joh} and currently receives
increased attention \cite{Davies, ER, BS2}.  \smallskip

The paper is organised as follows. We begin by recalling some material
from combinatorics and linear algebra in
Section~\ref{sec:prelim}. Then, the MCCP is described in more detail
in Section~\ref{sec:mccp}, with some emphasis on the defining
properties and their consequences. The corresponding families of
Markov matrices and generators are analysed in
Section~\ref{sec:matrices}, which establishes the core for the
embedding problem. Section~\ref{sec:inter} then dives a little deeper
into the algebraic, combinatorial and analytic structures via the
underlying moduli spaces for Markov matrices and generators. Finally,
Section~\ref{sec:pos} analyses the positivity conditions needed for
embeddability, which leads to a clear answer for this family of
processes.

\section{Preliminaries}\label{sec:prelim}

Below, we need some rather classic notions and results from
combinatorics and linear algebra. Since they will be combined in a
somewhat unusual manner, we recall a few properties of partitions and
matrices, while introducing our notation at the same time. While
we concentrate on cases that we need, our formulation is chosen with
an eye on potential generalisations to similar problems for more
general order lattices.

Let $S$ be a finite set, and consider the lattice $\cP (S)$ of
partitions of $S$; see \cite{Aigner} for background.  Here, we write a
partition of $S$ as $\cA = \{ A_{1}, \dots , A_{m} \}$, where
$m = |\cA|$ is the number of its (non-empty) parts (also called
blocks), and one has $A_{i} \cap A_{j} = \varnothing$ for all $i\ne j$
together with $A_{1} \cup \dots \cup A_{m} = S$.  Below, we shall need
a specific combinatorial identity, which we state and prove for lack
of reference and convenience of the reader.

\begin{fact}\label{fact:part-sum}
  If\/ $S$ is a non-empty finite set and\/ $\cP (S)$ its partition
  lattice, one has the identity
\[
  \sum_{\cA \in \cP (S)} (-1)^{\lvert \cA \rvert} \lvert
    \cA \rvert \ts !  \, = \, (-1)^{\lvert S \rvert} .
\]   
\end{fact}

\begin{proof}
  This can be proved by induction in the cardinality of $S$. When
  $\lvert S \rvert = 1$, there is only the trivial partition of $S$,
  and both sides of the formula evaluate to $-1$.
   
  Assume that the identity holds for $\lvert S \rvert = n$ and
  consider the set $S' = S \cup \{ e \}$ with one new element, $e$,
  which now allows the following induction argument. The partitions of
  $S'$ come in two types. Either, $\{ e \}$ forms a part of its own
  and thus augments a partition of $S$ by this part, or we have a
  partition of $S$ where $e$ is added to one of its existing parts,
  for which there are as many choices as there are parts.
   
  This means that we can evaluate our new sum via two separate 
  sums over $\cP (S)$ as
\[
\begin{split}
  \sum_{\cB \in \cP (S')} (-1)^{\lvert \cB \rvert} \lvert \cB \rvert
  \ts !  \, & = \! \sum_{\cA\in\cP (S)} \! (-1)^{\lvert \cA \rvert +
    1} \bigl( \lvert \cA \rvert +1 \bigr) !  \;\, + \!
  \sum_{\cA \in \cP (S)} \!  (-1)^{\lvert \cA \rvert} \lvert \cA \rvert
      \ts ! \cdot \lvert \cA \rvert \\[1mm]
  & = \, - \sum_{\cA\in\cP(S)} (-1)^{\lvert \cA \rvert} \lvert \cA
      \rvert \ts !  \, = \, - (-1)^{\lvert S \rvert} \, = \,
      (-1)^{\lvert S' \rvert} ,
\end{split}   
\]   
which completes the induction step and thus the proof.
\end{proof}

The lattice most relevant to us here is the family of subsets of
a given finite set, with the partial order of set inclusion. Here,
we consider $S=\{1, 2, \ldots , N \}$ and its $2^N$ subsets, which
form the power set lattice, denoted by $2^{S}$, where we use
$\subseteq$ for the partial order between two sets. Here,
$I \subset J$ is then used for $I \subseteq J$ with $I\ne J$, and we
write $J-I$ for $J\setminus I$ whenever $I\subseteq J$ for
simplicity. Further, $\lvert K \rvert$ denotes the cardinality of $K$,
and $\,\overline{\! K} = S - K$ is the complement of $K$ in $S$.

The M\"{o}bius function $\mu_{_\mathrm{M}}$ of this lattice is given by
$\mu_{_\mathrm{M}} (I,J) = (-1)^{\lvert J-I\rvert}$ when $I\subseteq J$ 
and by $\mu_{_\mathrm{M}} (I,J)=0$ otherwise. Further, the M\"{o}bius 
inversion formula acts as
\begin{equation}\label{eq:Mobius-1}
  g(K) \, = \sum_{I\subseteq K} f(I) \quad \Longleftrightarrow \quad
  f(K) \, = \sum_{I\subseteq K} (-1)^{\lvert K-I\rvert} g(I) \ts ,
\end{equation}
where $f$ and $g$ are arbitrary functions on $2^{S}$; see
\cite[Ch.~IV.2 and Ex.~4.15.IV]{Aigner} for more.

There is an interesting connection between the lattice $2^S$ of
subsets of a non-empty finite set $S$ and the partition lattice of $S$
as follows, where summation variables are marked with a dot below them
for clarity. For our later calculations with logarithms, we
choose a formulation that shows the interplay between multiplicative
and additive structures.

\begin{lemma}\label{lem:set-to-part}
  Let\/ $S$ be a non-empty finite set and let\/
  $f : 2^{S} \setminus \varnothing \longrightarrow \RR_{+}$ be a
  strictly positive function. Assume that\/ $f$ is extended to a
  function on the partitions of\/ $S$ by setting\/
  $f(\cA) = \prod_{A\in\cA} f(A)$ for any\/ $\cA\in \cP(S)$. Then, one
  has the identity
\[
   \sum_{\udo{\cA}\in \cP(S)} (-1)^{\lvert \cA \rvert -1}
   \bigl( \lvert \cA \rvert - 1 \bigr)  ! \, \log  f(\cA)
   \, = \,  \! \sum_{\varnothing \ne \udo{A} \subseteq S}
     \! (-1)^{\lvert S - A \rvert } \log f(A) \ts .
\]  
\end{lemma}

\begin{proof}
  Observing that $\log f (\cA) = \sum_{A\in\cA} \log f(A)$, the order
  of the double summation on the left-hand side can be changed, which
  turns it into
\[
   \sum_{\varnothing \ne \udo{A} \subseteq S}  \log f(A)
   \sum_{A \in \udo{\cA} \in \cP (S)} (-1)^{\lvert \cA \rvert - 1}
     \bigl( \lvert \cA \rvert - 1 \bigr) !
\]
Here, the second sum is $1$ when $A=S$, and otherwise turns into a sum
over all partitions of the form $\{ A, B_{1}, \ldots , B_{m} \}$,
where $\cB = \{ B_1, \ldots , B_m \}$ is a partition of
$S-A$. Adjusting the count of the number of parts, the second sum then
becomes
\[
   \sum_{\cB \in \cP (S-A)}  (-1)^{\lvert \cB \rvert} \lvert \cB \rvert
    \ts !    \, = \, (-1)^{\lvert S-A \rvert} 
\]  
   by Fact~\ref{fact:part-sum}. Putting the two cases
   together establishes the claimed identity.
\end{proof}

Some of our results would naturally be formulated in terms of the
(complex) \emph{Jordan normal form} (JNF) of a given matrix. Since we
will later see that we only need to deal with diagonalisable matrices,
we simplify the tools and methods to this setting. Such matrices have
a JNF with only trivial elementary Jordan blocks. We call a matrix $B$
\emph{simple} if its eigenvalues are distinct, in which case $B$ is
automatically diagonalisable. When we have multiple  (or
repeated) eigenvalues, we call $\sigma (B)$, the spectrum of $B$,
\emph{degenerate}.

Markov matrices have a special spectral structure as follows, which we
recall from \cite{BS3}. Throughout, we use $\cM_d$ to denote the
subset of Markov matrices in $\Mat (d,\RR)$, which is a closed convex
set, and $\AAA^{\! (0)}_{d}$ for the set of all matrices from
$\Mat (d, \RR)$ with zero row sums. They form a non-unital algebra,
because $\one$ is not an element of it, and no other two-sided unit
exists in it. All Markov generators are elements of
$\AAA^{\! (0)}_{d}\! $. Let us first recall \cite[Fact~2.2]{BS3}.

\begin{fact}\label{fact:one-is-special}
  For all\/ $M\in\cM_d$, one has\/ $1\in\sigma (M)$ together with
  equal algebraic and geometric multiplicity. In particular, there is
  no non-trivial Jordan block for\/ $\lambda = 1$.
  
  Further, the corresponding statement holds for generators, which is
  to say that any Markov generator has\/ $0$ as an eigenvalue, again
  with no non-trivial Jordan block for it.  \qed
\end{fact}

Given a real matrix $B$, we need to know when it possesses a real
matrix logarithm, that is, a real matrix $R$ such that $B = \ee^R$
holds, where the exponential of a matrix is defined by the convergent
series $\ee^R = \sum_{n=1}^{\infty} \frac{1}{n\ts !} R^n$. Further, we
ask when such a real logarithm is unique. The following important
characterisation follows from \cite[Thms.~1 and 2]{Culver}.

\begin{fact}\label{fact:Culver}
  A diagonalisable matrix\/ $B\in\Mat (d,\RR)$ has a real logarithm if
  and only if the following two conditions are satisfied.
\begin{enumerate}\itemsep=2pt
\item The matrix\/ $B$ is non-singular.
\item Each negative eigenvalue of\/ $B$ occurs with even multiplicity.
\end{enumerate}   
Further, when\/ $B$ has simple spectrum with positive eigenvalues
only, the real logarithm of\/ $B$ is unique.  \qed
\end{fact}

Finally, we need the \emph{spectral mapping theorem} (SMT). When $B$
is a matrix with spectrum $\sigma (B)$, and $h (x)$ is any polynomial
in one variable, also $h (B)$ is well defined, and the general SMT
\cite[Thm.~10.33]{Rudin} applied to matrices states that
\begin{equation}\label{eq:SMT}
  \sigma \bigl( h (B) \bigr) \, = \, h \bigl( \sigma (B) \bigr)
  \, \defeq \, \{ \ts h (\lambda) : \lambda \in \sigma (B) \} \ts ,
\end{equation} 
which holds for the spectrum including multiplicities. This extends
easily to power series of matrices when all eigenvalues lie inside the
disk of convergence. This is all we shall need below; see
\cite{Higham, HJ} for further background.

\section{The multiple coupon collection process}\label{sec:mccp}

Here, with $S = \{ 1, 2, \ldots , N\}$ for the $N$ objects, we
use the corresponding power set lattice, $2^{S}$, as introduced
above. Now, let $Z\subseteq S$ be the (random) set of objects sampled
in one step, and let its distribution be given by
$p = (p^{\pa}_{K})^{\pa}_{K \subseteq S}$, that is,
$\PP (Z=K) = p^{\pa}_{K} \geqslant 0$ for $K\subseteq S$, with
$\sum_{K\subseteq S} p^{\pa}_{K} = 1$.  As mentioned in the introduction,
the MCCP then is the discrete-time Markov chain
$(X^{\pa}_n)^{\pa}_{n\in\NN_0}$ on the set of subsets of $S$, where $X^{\pa}_n$
is the set of distinct objects sampled until step $n$. Note that no
record is kept on the number of times an object occurs in the
process. Below, we assume $X^{\pa}_{0} = \varnothing$ unless stated
otherwise.  Clearly, $X^{\pa}_{n+1} = X^{\pa}_{n} \cup I$ with probability
$p^{\pa}_{I}$ for $I\subseteq S$, independently for every $n$ and
independently of the current state $X^{\pa}_n$.  This gives the
transition probabilities
$M^{\pa}_{IJ} = \PP (X^{\pa}_{n+1} = J \, | \, X^{\pa}_{n} = I )$ as
\begin{equation}\label{eq:def-prop-PM}
     M^{\pa}_{IJ} \, = \sum_{ I \cup \udo{K} = J } p^{\pa}_{K}
\end{equation}
for all $I, J \subseteq S$, where $K$ runs through all subsets of $J$
subject to the given condition. The relation \eqref{eq:def-prop-PM}
includes the relation $M^{\pa}_{IJ} = 0$ for $I\not\subseteq J$
because the empty sum is zero by convention. In particular, $M$ is
triangular, with $M^{\pa}_{SS} = 1$.  Let us formalise this as follows.

\begin{definition}\label{def:PM}
  Let\/ $S=\{ 1, 2, \ldots , N \}$, for an arbitrary but fixed\/
  $N\in\NN$.  Then, a matrix\/ $M \in \Mat (2^N, \RR)$ is said to
  possess \emph{Property} CM (coupon Markov), or to be a 
  CM matrix for short, if there
  is a probability vector\/ $p = (p^{\pa}_{K})^{\pa}_{K\subseteq S}$ such
  that Eq.~\eqref{eq:def-prop-PM} holds for all\/ $I,J \subseteq S$.
\end{definition}

Despite the origin of \eqref{eq:def-prop-PM} from the MCCP,
Definition~\ref{def:PM} does not require $M$ to be a Markov matrix in
the first place. The Markov property should then be viewed as a
consequence of the parametrisation. Indeed, we have the following
simple consistency result.

\begin{fact}\label{fact:easy}
   Any CM matrix is a Markov matrix.
\end{fact}

\begin{proof}
  If $p$ is a probability vector, we have $p^{\pa}_{K}\geqslant 0$ for
  all $K\subseteq S$, and then $M^{\pa}_{IJ} \geqslant 0$ for all
  $I,J\subseteq S$ by \eqref{eq:def-prop-PM}. Since we also have
  $\sum_{K\subseteq S} p^{\pa}_{K} = 1$, we can calculate the row sums of
  $M$ as
 \[
     \sum_{\udo{J}\subseteq S} M^{\pa}_{IJ} \, = \sum_{\udo{J}\subseteq S}\,
     \sum_{I \cup \udo{K} = J} p^{\pa}_{K} \, = \sum_{I\subseteq\udo{J}\subseteq S}
     \,  \sum_{\udo{L}\subseteq I} p^{\pa}_{(J-I)\dot{\cup} L} 
      \, = \sum_{\udo{K}\subseteq S} p^{\pa}_{K}   \, = \, 1  \ts ,
 \]
   which is true for any $I\subseteq S$.
\end{proof}

The class of CM matrices precisely consists of the MCCP Markov
matrices.  Clearly, Property CM is preserved under convex
combinations, as we shall exploit below and in
Section~\ref{sec:inter}. Due to the bijective parametrisation by
probability vectors, the set of CM matrices, for any fixed set
$S$, has local dimension $2^{\lvert S \rvert} \nts - 1$. All such
matrices have a rather specific triangular form, and can be
interpreted as elements of the \emph{incidence algebra} \cite{Aigner,
  Spiegel} of $2^{S}$, which manifests itself in strong properties due
to the underlying lattice structure.

\begin{example}\label{ex:reco}
  Let us look at the elementary special case of each object $i \in S$
  appearing independently with probability $\pi_{i}$, where we assume
  $0 < \pi_i < 1$ to avoid degeneracies. Let us mention in passing
  that this case also appears naturally in the context of genetics,
  where the objects are the links between the sites in a sequence,
  each of which is cut by genetic recombination, independently in
  every generation and independently of each other \cite{EMB}.
 
  For $K\subseteq S$, this gives
\[
    p^{\pa}_{K} \: = \, \Bigl( \, \prod_{i\in K} \pi_i \Bigr) 
    \prod_{j \in \, \overline{\! K}} (1-\pi_j)
\] 
 and thus the discrete-time Markov matrix $M$ with entries
\begin{equation}\label{eq:reco}
  M^{\pa}_{IJ} \, = \, \PP ( X^{\pa}_{1} = J \, | \, X^{\pa}_{0} = I )
  \: = \, \Bigl( \, \prod_{i \in J-I}  \! \pi_{i} \Bigr)
  \prod_{j\in S-J} (1 - \pi_{j} ) \ts ,
\end{equation}
where $(X_n )^{\pa}_{n\in \NN_{0}}$ is the corresponding discrete-time
Markov chain. The derivation of \eqref{eq:reco} employed the identity
$1 = \prod_{i\in I} \bigl( \pi_{i} + (1 - \pi_{i}) \bigr) =
\sum_{K\subseteq I} \bigl(\prod_{i\in K} \pi_{i} \bigr) \prod_{j\in
  I-K} (1-\pi_{j})$. \smallskip
  
The corresponding continuous-time process, by an educated guess or
insight from the existing literature, consists of independent Poisson
processes, with rate $r_{i} = - \log (1 - \pi_i )$ at site $i\in
S$. This gives $\ee^{-r_i} = 1-\pi_i$ and describes the appearance of
links (or coupons) at time $1$ via the probabilities
\[
  \PP (Y^{\pa}_{1} = J \, | \, Y^{\pa}_{0} = I ) \: = \, \Bigl( \,
  \prod_{i\in J-I} (1-\ee^{-r_{i}}) \Bigr) \prod_{j \in S - J}
  \ee^{-r_{j}} ,
\]   
where $\bigl( Y^{\pa}_{t} \bigr)_{t \geqslant 0}$ now is the Markov chain
in continuous time. This matches the discrete-time formula from
\eqref{eq:reco}.  Consequently, in this non-degenerate case of
independent sampling, we always have embeddability. \exend
\end{example}

While this example looks simple, the connection between discrete and
continuous time processes is considerably more complex in the presence
of dependencies, and requires a more detailed investigation.

\section{Markov matrices and generators
  for MCCPs}\label{sec:matrices}

Let us analyse the structure of CM matrices in some detail, and then
derive the corresponding notion for Markov generators (and
beyond). This will establish a rather strong algebraic structure that
is extremely helpful for dealing with the embedding problem.

\begin{lemma}\label{lem:prop-PM}
  If\/ $M$ is a\/ \textnormal{CM} matrix, with parametrising
  probability vector\/ $p$, any power\/ $M^n$ with\/ $n\in\NN$ is a\/
  \textnormal{CM} matrix as well, then with parameter vector\/
  $p^{(n)} = p M^{n-1}$, whose entries are
\[
  p^{(n)}_{\ts\ts K} \, = \, \bigl( p M^{n-1}\bigr)_{K} \, = \!
  \sum_{\substack{U^{\pa}_{1}, \ldots , U^{\vphantom{I}}_{n} \\
      U^{\pa}_{1} \nts \cup \dots \cup \ts U^{\pa}_{n} = K}}
      \prod_{i=1}^{n} p^{\pa}_{\ts U_{i}} \ts .
\]
\end{lemma}

\begin{proof}
  The probability vector $p$ coincides with the row
  $(M^{\pa}_{\varnothing K})^{\pa}_{K\subseteq S}$ of the Markov matrix
  $M$. The corresponding row of $M^n$ clearly is
  $p^{(n)} = p M^{n-1}$, and we can inductively calculate
\begin{equation}\label{eq:step}
\begin{split}
   p^{(n+1)}_{\, K} \, & = \, \bigl( p^{(n)} M \bigr)_{K} \, = 
   \sum_{I \subseteq K} p^{(n)}_{\, I} M^{\phantom{(n)}}_{IK}
   \, = \sum_{I\subseteq K}
   \sum_{\substack{U^{\pa}_{1}, \ldots , U^{\pa}_{n} \\
      U^{\pa}_{1} \nts \cup \dots \cup \ts U^{\vphantom{I}}_{n} = I}}
     \biggl( \, \prod_{i=1}^{n} p^{\pa}_{\ts U_{i}} \biggr)
     \sum_{I \cup \udo{J} = K} p^{\pa}_{J} \\
   & = \sum_{\substack{ U^{\pa}_{1}, \ldots , U^{\pa}_{n}, J \\
       U^{\pa}_{1} \nts \cup \dots \cup \ts U^{\vphantom{I}}_{n} \cup J = K}}
       \!\! p^{\pa}_{J} \; \prod_{i=1}^{n} p^{\pa}_{\ts U_i} \, = 
       \sum_{\substack{ U^{\pa}_{1}, \ldots , U^{\pa}_{n+1}\\
       U^{\pa}_{1} \nts \cup \dots \cup \ts U^{\vphantom{I}}_{n+1} = K}}
         \prod_{i=1}^{n+1} p^{\pa}_{\ts U_i} \ts ,
\end{split}       
\end{equation}  
   which establishes the claimed formula for the $p^{(n)}$.
   
   It remains to prove that $p^{(n)}$ is the probability vector needed
   for the validity of Property CM for $M^n$. This is clear for $n=1$
   by assumption, so we can once again proceed inductively from $n$ to
   $n + 1$, this time via
\[
\begin{split}
   M^{n+1}_{\, IJ} \,  & = \sum_{I\subseteq \udo{K}\subseteq J}
   M^{\ts n}_{I K} \, M^{\phantom{(n)}}_{K J} \, = 
   \sum_{I \subseteq \udo{K} \subseteq J} \,
   \sum_{I \cup \udo{H} = K} p^{(n)}_{\ts H} \!
       \sum_{K\cup \ts \udo{U}=J}
   p^{\phantom{(n)}}_{\ts U} \\[2mm]
   & = \sum_{\substack{H, U \subseteq J \\ 
     I \cup H \cup \ts U^{\vphantom{I}} = J}} \!\!
     p^{(n)}_{\ts H}  p^{\phantom{(n)}}_{\ts U}  =
     \sum_{I \cup \udo{K} = J}  p^{(n+1)}_{\, K} ,
\end{split}
\]
where the last step used the relations for the $p^{(n+1)}$ derived 
in \eqref{eq:step}.
\end{proof}

Next, let us look at the property analogous to CM at the level of rate
matrices. To do so, we consider $A=M\nts - \one$, which is a rate
matrix in the incidence algebra (its off-diagonal entries are
non-negative and its row sums are zero, as are all elements
$A^{\pa}_{IJ}$ with $I\not\subseteq J$). What is more, it inherits some
of Property CM as follows. For $I\ne J$, we have
$A^{\pa}_{IJ} = M^{\pa}_{IJ}$ and Eq.~\eqref{eq:def-prop-PM} applies, while
the diagonal elements are fixed by the zero row sum condition, so
$A^{\pa}_{I\ts I} = - \sum_{I\subset \udo{J}} A^{\pa}_{IJ}$ holds for all
$I\subseteq L$.

This suggests that, given $S$, we can make a (preliminary) definition
of what we will call \emph{Property} CG (coupon generator), for
all elements $B$ of the incidence algebra with zero row sums, by
saying that a real vector
$r = (r^{\pa}_{K})^{\pa}_{\varnothing\ne K\subseteq S}$ of dimension
$2^{\lvert S \rvert} - 1$ exists such that
\begin{equation}\label{eq:def-prop-PG}
\begin{split}
    B^{\pa}_{I J} \, & = \sum_{I\cup \udo{K} = J} r^{\vphantom{\chi}}_{K}
    \qquad \text{for $I\ne J$ together with} \\
    B^{\pa}_{I\ts I} \, & = \, - \sum_{I \subset \udo{J}} B^{\pa}_{I J}
\end{split}
\end{equation}
holds for all $I, J \subseteq S$. Then, $A = M\nts - \one$ satisfies
Property CG with $r^{\pa}_{K} = p^{\pa}_{K}$ for
$\varnothing\ne K \subseteq S$, and all row sums are zero.  Note that
Property CG is a \emph{linear} property in the sense that any linear
combination of CG matrices is another matrix with this property. In
other words, the CG matrices form a vector space over $\RR$. Note that
this property is \emph{not} restricted to rate matrices. A little
later, we shall see how one can consistently add a value for
$r^{\pa}_{\varnothing}$, then giving access to more powerful algebraic
methods, but we do not need this for our present discussion and thus
first continue with the conditions from \eqref{eq:def-prop-PG}.

\begin{lemma}\label{lem:A-and-powers}
  Let\/ $M$ be a\/ \textnormal{CM} matrix and set\/ $A=M\nts -
  \one$. Then, for every\/ $k\in\NN$, the matrix\/ $A^k$ has zero row
  sums and possesses Property\/ \textnormal{CG} according to
  \eqref{eq:def-prop-PG}, with a unique real parameter vector\/
  $r^{(k)}$ of dimension\/ $2^{\lvert S \rvert} - 1$.
\end{lemma}

\begin{proof}
  It is clear that $A$ has zero row sums, and writing $A^k = A^{k-1} A$
  for $k>1$ can be used to verify this property for all $A^k$.

  Property CG was derived above for the matrix $A$, when we motivated
  the notion. It is also satisfied by every matrix of the form
  $M^k \nts - \one$, by Lemma~\ref{lem:prop-PM}, with a parameter
  vector that derives from $p^{(k)}$ in the obvious way. From here, we
  get Property CG for $A^k$ via
\[
   A^k \, = \sum_{m=1}^{k} \binom{k}{m} (-1)^{k-m}
        \bigl( M^m \nts - \one \bigr) ,
\]
which easily follows from a double application of the binomial
formula, and the linearity of Property CG mentioned earlier. The
uniqueness of $r^{(k)}$ follows from the relation
$r^{(k)}_{\ts\ts K} = (A^{k})_{\varnothing K}$ for
$K\nts\ne \varnothing$.
\end{proof}

At this point, we can state the situation for CM matrices as follows.

\begin{prop}\label{prop:real-log-PG}
  If\/ $M$ is a non-singular\/ \textnormal{CM} matrix, it possesses a
  real logarithm, $R$, which is an element of the incidence algebra
  with zero row sums and real spectrum.
  
  Further, $R$ possesses Property\/ \textnormal{CG} as in
  \eqref{eq:def-prop-PG} with a unique, real parameter vector\/ $r$,
  and\/ $R$ is a Markov generator if and only if\/ $r$ has
  non-negative entries only.
\end{prop}

\begin{proof}
  Being a CM matrix implies that all eigenvalues of $M$, which are its
  diagonal elements, are non-negative numbers and bounded by
  $1$. Since $\det (M) \ne 0$, we obtain $\sigma (M) \subset
  (0,1]$. With $A=M \nts -\one$, one real logarithm can be given as
  the principal matrix logarithm \cite{Higham} via the power series
\begin{equation}\label{eq:log}
   R \, = \, \log (\one + A) \, = \sum_{k=1}^{\infty}
   \myfrac{(-1)^{k-1}}{k} A^k ,
\end{equation}  
which is convergent because the spectral radius of $A$ satisfies
$\varrho^{\pa}_{A} < 1$. The matrix $R$, by an application of the
Cayley--Hamilton theorem, is a linear combination of $A$ and finitely
many of its (positive) powers, so clearly a matrix with zero row sums.

Due to the triangular structure of $R$, which is inherited from that
of $A$ and its powers, all eigenvalues of $R$ are diagonal elements
and thus real. Further, $R$ is a CG matrix, by 
Lemma~\ref{lem:A-and-powers}, with a real parameter vector $r$ that is
unique. Since $r^{\pa}_{K} = R^{\pa}_{\varnothing K}$ for
$K\ne \varnothing$, we see that $R$ can only be a Markov generator if
all these entries are non-negative. If so, Eq.~\eqref{eq:def-prop-PG}
then implies that all off-diagonal elements of $R$ are non-negative as
well.
\end{proof}

We are now in the position to further analyse the embedding problem,
which is significantly simplified in the sense that it suffices to
check the entries of a vector rather than all off-diagonal elements of
$R$.  To do so, we can now relate the parameters of $M$ and $R$ with
the corresponding eigenvalues, and then the latter with each other via
the SMT from Eq.~\eqref{eq:SMT}.

Indeed, given a non-singular CM matrix, its
eigenvalues are labelled with $K\subseteq S$ and read
\begin{equation}\label{eq:def-r}
  \lambda^{\pa}_{K} \, = \, M^{\pa}_{KK} \, =
  \sum_{K\cup \udo{I} = K} p^{\pa}_{I} \, =
  \sum_{\udo{I}\subseteq K} p^{\pa}_{I}
\end{equation}
by an application of Eq.~\eqref{eq:def-prop-PM}, where
$\lambda^{\pa}_{\varnothing} =
M^{\pa}_{\varnothing\varnothing}=p^{\pa}_{\varnothing}$ and
$\lambda^{\pa}_{S} = M^{\pa}_{SS}=1$. At this point, we can say more about
the matrix structure as follows.

\begin{prop}\label{prop:diag}
  Every\/ \textnormal{CM} matrix is diagonalisable. All such matrices,
  for any fixed set\/ $S$, commute with one another.
\end{prop}

\begin{proof}
  Consider the column vectors
  $e^{I} \defeq (\delta^{\pa}_{I, J})^{\pa}_{J\subseteq S}$ with
  $I\subseteq S$, which define the standard basis of $\RR^d$ with
  $d=2^{\lvert S \rvert}$. Now, set
\[
  v^{K} \, = \sum_{\udo{I} \subseteq K} \! e^{I} \qquad \text{with }
  K \subseteq S \ts ,
\]
which are linearly independent and thus also constitute a basis. We
will now show that, when $M\in \Mat (d,\RR)$ has Property CM (so 
$M\in\cM_d$ by Fact~\ref{fact:easy}), we get
$M v^K = \lambda^{\pa}_{K} v^K$ for all $K\subseteq S$, independently of
the parameter vector $p$ of $M$. Consequently, we always have equal
algebraic and geometric multiplicities of the eigenvalues, and $M$ is
diagonalisable.

The claim follows from an explicit calculation. For $I,K\subseteq S$,
we have
\[
\begin{split}
  (M v^K )^{\pa}_{I} \, & = \sum_{\udo{J} \subseteq S} M^{\pa}_{IJ}
  \Bigl( \, \sum_{\udo{L} \subseteq K} e^L \Bigr)_J \, = \,
  \sum_{\udo{J} \subseteq S} M^{\pa}_{IJ} \sum_{\udo{L} \subseteq K}
  \delta^{\pa}_{L, J} \\[2mm]
  & = \! \sum_{I \subseteq \udo{L} \subseteq K} \! M^{\pa}_{IL} 
  \, = \! \sum_{I \subseteq \udo{L} \subseteq K} \, \sum_{I\cup \udo{U} = L}
  p^{\pa}_{U} \, = \! \sum_{I \subseteq \udo{L} \subseteq K}
  \, \sum_{\udo{V}\subseteq I} p^{\pa}_{(L-I)\cup V} \ts ,
\end{split}
\]
where we have dropped terms that are zero in the third step, and
then used \eqref{eq:def-prop-PM}. Now, the last expression is
zero whenever $I\not\subseteq K$. Otherwise, $I\subseteq K$ 
and the expression equals
\[
  \sum_{\udo{U} \subseteq K-I} \, \sum_{\udo{V}\subseteq I}
  p^{\pa}_{\ts U\cup V} \, = \sum_{\udo{J} \subseteq K} p^{\pa}_{J}
  \, = \, \lambda^{\pa}_{K}
\]
by \eqref{eq:def-r}. The two cases together establish that
$(M v^K)^{\pa}_{I} = (\lambda^{\pa}_{K} v^K )^{\pa}_{I}$ holds for all
$I,K \subseteq S$, and $v^K$ is indeed an eigenvector for
$\lambda^{\pa}_{K}$ as claimed.

Consequently, for fixed $S$,  all CM matrices are 
\emph{simultaneously} diagonalisable by the same 
parameter-independent basis matrix, and hence commute.
\end{proof}

\begin{remark}\label{rem:commute}
  Given $S$ and $d=2^{\lvert S \rvert}$, the simultaneous
    diagonalisability of all CM matrices due to
    Proposition~\ref{prop:diag} also implies that they are
    multiplicatively closed and hence form a commutative monoid within
    $\cM_d$, as we shall analyse in more detail in
    Section~\ref{sec:inter}. This property can also be seen as
    follows. For $K\subseteq S$, define the matrix $M^{(K)}$ via
\[
    M^{(K)}_{IJ} \, = \, \begin{cases} 1 , & \text{if $I\cup K = J$}, \\
        0 , & \text{otherwise} , \end{cases}
\]  
which is the CM matrix with parameter vector
$p^{(K)} = (\delta^{\pa}_{K,L})^{\pa}_{L \subseteq S}$. The matrices
$M^{(K)}$ are the $d$ extremal elements among the CM matrices, and
satisfy
\begin{equation}\label{eq:extremal}
     M^{(K)} M^{(L)} \, = \, M^{(K\cup L)} \, = \, M^{(L)} M^{(K)} .
\end{equation}  
Since every CM matrix is a convex combination of the extremal ones,
commutativity of all CM matrices follows. Various other properties can
also be derived from \eqref{eq:extremal}.  \exend
\end{remark}

From Eq.~\eqref{eq:def-r}, it is also clear that we have
\begin{equation}\label{eq:det}
  \det (M) \ne 0 \quad \Longleftrightarrow \quad
  p^{\pa}_{\varnothing} > 0 \ts ,
\end{equation}
which has the following immediate consequence.

\begin{coro}\label{coro:no-go}
  When\/ $M\in \cM_d$ has Property\/ \textnormal{CM} with\/
  $p^{\pa}_{\varnothing} = 0$, it is not embeddable, neither into
  a time-homogeneous semigroup nor into a time-inhomogeneous
  Markov flow.
\end{coro}

\begin{proof}
  Both types of embeddability require $M$ to be non-singular, via
  $\det (\ee^Q) = \ee^{\mathrm{tr} (Q)}$ and Eq.~\eqref{eq:det} in the
  first case and an application of Liouville's theorem to the
  Kolmogorov forward equation, $\dot{M}(t) = M(t) \ts Q (t)$, in the
  latter; compare \cite[Rem.~3.3]{BS4}.
\end{proof}

When $p^{\pa}_{\varnothing} > 0$, we have
$\sigma (M) \subset (0,1]$, while all eigenvalues of $R$ from
\eqref{eq:log} are still real. We can thus employ the SMT with the
standard real logarithm (which is the unique inverse of the
exponential function as a mapping from $\RR$ to $\RR_{+}$) to
calculate the eigenvalues of $R$ as
\[
  \mu^{\pa}_{K} \, = \, \log (\lambda^{\pa}_{K}) \, = \, R^{\pa}_{KK} \ts ,
  \qquad \text{with } K\subseteq S \ts .
\]
Here, $\mu^{\pa}_{S}=0$ reflects the zero row sum property of $R$.  Under
our assumptions, Proposition~\ref{prop:real-log-PG} implies the
existence of a real parameter vector
$(r^{\pa}_{K})^{\pa}_{\varnothing\ne K \subseteq S}$ so that
$R^{\pa}_{IJ} = \sum_{I\cup \udo{K}=J} r^{\pa}_{K}$ for $I\ne J$ together
with $R^{\pa}_{I \ts I} = - \sum_{I \subset \udo{J}} R^{\pa}_{IJ}$.  Since
$R$ still is a triangular matrix, we thus also get the relation
\begin{equation}\label{eq:previous}
    \mu^{\pa}_{K} \, = \, - \sum_{K\subset \udo{H}} R^{\pa}_{KH} \, =
    \, - \sum_{K\subset \udo{H}} \, \sum_{K\cup \ts\udo{I}=H}
    \! r^{\vphantom{\chi}}_{I}
    \, = \, - \! \sum_{\udo{I} \cap \, \overline{\! K} \ne \varnothing}
    \! r^{\vphantom{\chi}}_{I} \ts .
\end{equation}
It can be solved for the entries of $r$ via the M\"{o}bius inversion
formula from \eqref{eq:Mobius-1} and a few extra steps (omitted here
because we shall see a simpler approach in Remark~\ref{rem:simpler}),
which gives
\begin{equation}\label{eq:magic}
    r^{\vphantom{\chi}}_{K} \, =
    \sum_{\varnothing \ne \udo{H} \supseteq \,\overline{\! K}}
    (-1)^{\lvert H - \,\overline{\! K} \ts\rvert} \ts
      \mu^{\pa}_{\, \overline{\! H}} \ts\ts ,
    \qquad \text{with } \mu^{\pa}_{I} = \log(\lambda^{\pa}_{I}) \ts 
    \text{ and } K\ne \varnothing \ts .
\end{equation}
We have thus arrived at the following result.

\begin{theorem}\label{thm:log-relation}
  If\/ $M$ is a non-singular\/ \textnormal{CM} matrix, it has a real
  logarithm, $R=\log(\one + A)$, which satisfies Property\/
  \textnormal{CG} with the real parameter vector\/ $r$ from
  \eqref{eq:magic}.  This\/ $R$ is a Markov generator if and only if
  all entries of\/ $r$ are non-negative, in which case\/ $M$ is
  embeddable.

  Further, when\/ $M$ has simple spectrum, $R$ is the only real
  logarithm of\/ $M\nts$, and\/ $M$ is embeddable into a Markov
  semigroup if and only if\/ $R$ is a Markov generator.
\end{theorem}

\begin{proof}
  The first part is a consequence of Proposition~\ref{prop:real-log-PG}
  together with our above calculations around the eigenvalues.
 
  The final claim is clear once we know that $R$ is the only real
  logarithm of $M$, which is the case when the spectrum of $M$ is
  simple, by Fact~\ref{fact:Culver}, but not in general.
\end{proof}

Formally inserting $K=\varnothing$ in \eqref{eq:magic} and observing
$\overline{\varnothing} = S$, one obtains
\[
    r^{\pa}_{\varnothing} \, = \, \mu^{\pa}_{\varnothing} \, = \,
    \log (\lambda^{\pa}_{\varnothing}) \, = \, \log (p^{\pa}_{\varnothing}) \ts ,
\]
which satisfies $r^{\pa}_{\varnothing} \leqslant 0$. From now on, we use
this definition to extend the parameter vector $r$ to one of full
dimension $2^{\lvert S \rvert}$.  Due to the zero row sum property of
$R$, one then has the relation $\sum_{K\subseteq S} r^{\pa}_{K} = 0$, as
is most easily seen by inserting $K=\varnothing$ into
\eqref{eq:previous}. Also, with this extension, we refine our previous
definition as follows.

\begin{definition}\label{def:prop-PG}
  Let\/ $S = \{ 1, 2, \ldots, N \}$ and let\/ $2^S$ be the 
  order lattice of the\/ $d=2^{N}$ subsets of\/ $S$.  A
  matrix\/ $B\in \Mat ( d, \RR )$, indexed with the elements of\/
  $2^S$, is said to have \emph{Property} CG, or to be a
  CG matrix, if there is a real vector\/ $r \in \RR^{d}$ with\/
  $\sum_{I\subseteq S} r^{\pa}_{I} = 0$ such that
\[
     B^{\pa}_{IJ} \, = \sum_{I\cup \udo{K} = J} r^{\pa}_{K}
\]
    holds for all\/ $I,J \subseteq S$, where empty sums are zero.
\end{definition}

Before we continue, we need to check that this definition is
consistent with our preliminary one from 
Eq.~\eqref{eq:def-prop-PG} above.  Indeed, in comparison with
\eqref{eq:def-prop-PG}, we see that we have the same type of condition
for all $B^{\pa}_{IJ}$ with $I\ne J$, where the element
$r^{\pa}_{\varnothing}$ never shows up. So, we have to verify that the
conditions for all $B^{\pa}_{II}$ match. From \eqref{eq:def-prop-PG}, in
analogy with \eqref{eq:previous}, we get
\[
  B^{\pa}_{II} \, = \, - \sum_{I\subset \udo{J}} B^{\pa}_{IJ} \, = \, -
  \sum_{I\subset\udo{J}} \, \sum_{I\cup\udo{K}=J} \! r^{\pa}_{K} \, =
  \, - \! \sum_{\udo{K} \cap \, \overline{\nts\nts I} \ts \ne
    \varnothing} \! r^{\pa}_{K} \, = \sum_{\udo{K}\nts\subseteq I}
  r^{\pa}_{K} \ts ,
\]
where the last step is a consequence of
$\sum_{K\subseteq S} r^{\pa}_{K} = 0$.  Definition~\ref{def:prop-PG}
simply gives
\[
    B^{\pa}_{II} \, = \sum_{I\cup\udo{K}=I} r^{\pa}_{K} \, = 
    \sum_{\udo{K} \subseteq I} r^{\pa}_{K} \ts ,
\]
which is the same sum. 

\begin{lemma}\label{lem:easy-2}
  If\/ $B$ is a\/ \textnormal{CG} matrix, it has zero row sums. If it
  is also a rate matrix, $\ee^B$ is a Markov matrix with Property\/
  \textnormal{CM}.
\end{lemma}

\begin{proof}
  The first claim is the analogue of Fact~\ref{fact:easy}, and has the
  same proof (with the parameter vector $p$ replaced by $r$). This
  means that $0 \in \sigma (B)$ with eigenvector
  $v = ( 1,1, \ldots , 1)^{\trans}$.

  If $B$ is a Markov generator, $\ee^B$ is a Markov matrix; this
  follows easily from \cite[Eq.~10.2]{Higham},
\[
     \ee^B \, = \lim_{n\to\infty} \bigl( \one + \tfrac{1}{n} B \bigr)^{n} ,
\]
by observing that, for all sufficiently large $n$, the matrix
$\one + \frac{1}{n} B$ has non-negative entries only. As $Bv=0$,
the required row sum condition is a consequence of $\ee^B v = v$.

Now, observe that $\ee^B = \one + A$, where $A$ is a CG matrix with
non-negative off-diagonal entries. It thus has a parameter vector $r$
with non-negative elements (except at index $\varnothing$). One
verifies that $\one + A$ satisfies Eq.~\eqref{eq:def-prop-PM} with
parameter vector
$p = ( r^{\pa}_{K} + \delta^{\pa}_{\varnothing, K})^{\pa}_{K \subseteq S}$,
which is a probability vector. This establishes Property CM for
$\ee^B$.
\end{proof} 

So, we now have a completely analogous way to define Properties CM and
CG, both with a real parameter vector that agrees with one row of the
matrix. The difference is the row sum, which is $1$ for CM and
$0$ for CG, as we know it from the row sums of the corresponding
matrices. Since Definitions~\ref{def:PM} and \ref{def:prop-PG} show
the same structure with respect to the parameter vector, we can 
repeat the arguments from Proposition~\ref{prop:diag} and
Remark~\ref{rem:commute} to get the following.

\begin{coro}\label{coro:diag}
  Any\/ \textnormal{CG} matrix is diagonalisable. All such matrices,
  for any fixed\/ $S$, commute with one another, and thus form an
  Abelian subalgebra of $\ts\AAA^{(0)}_{d}\! $, where\/
  $d = 2^{\lvert S \rvert}\nts $. \qed
\end{coro}

\begin{remark}\label{rem:simpler}
  The above derivation can now be used to calculate the semigroup
  $\{ \ee^{\ts t \ts Q} : t \geqslant 0 \}$ non-recursively, for any
  rate matrix $Q$ with Property CG. Indeed, if $r$ is the parameter
  vector for $Q$, the eigenvalues of $Q$ are its diagonal elements,
  $\mu^{\pa}_{K} = Q^{\pa}_{KK}$ with $K\subseteq S$, and we have the
  relations from Eq.~\eqref{eq:previous}, with $R$ replaced by $Q$. In
  fact, with Definition~\ref{def:prop-PG}, we simply get
\begin{equation}\label{eq:new-version}
     \mu^{\pa}_{K} \, = \, Q^{\pa}_{KK} \, = \sum_{\udo{I} \subseteq K} r^{\pa}_{I} 
     \qquad \text{and} \qquad r^{\pa}_{K} \, = \sum_{\udo{I} \subseteq K}
     (-1)^{\lvert K - I \rvert} \mu^{\pa}_{I} \ts ,
\end{equation}
which replaces Eq.~\eqref{eq:magic} and its (omitted) derivation.
Clearly, $t \ts Q$ with $t\geqslant 0$ still has Property CG, now with
parameter vector $t\ts r$ and eigenvalues $t \mu^{\pa}_{K}$.  Then, for
any $t\geqslant 0$, the eigenvalues of $M(t)=\ee^{\ts t \ts Q}$ are
$\lambda^{\pa}_{K} (t) = \ee^{\ts t \ts \mu^{\pa}_{K}}$, with
$K\subseteq S$.
   
If $Q$ has Property CG, the same applies to $Q^{\ts k}$ for $k\in\NN$,
which need not be rate matrices for $k>1$, though they are always
elements of $\AAA^{(0)}_{d}$. Indeed, profiting from the consistent
definition of $r^{\pa}_{\varnothing}$, we can now mimic our arguments for
Lemma~\ref{lem:prop-PM} to see that $r^{(2)} = r \ts Q$ is the
parameter vector for $Q^2$ and we get those of $Q^{\ts k}$ recursively
via $r^{(k+1)}=r^{(k)} Q$ for $k\geqslant 1$.  Then, the exponential
series in conjunction with the Cayley--Hamilton theorem tells us that
$\ee^{\ts t \ts Q} = \one + A(t)$ holds for every $t\geqslant 0$,
where $A(t)$ is a polynomial in $Q$ and its (positive) powers, hence a
CG matrix.  This implies that each $M(t)$ is a CM matrix, with a
unique parameter vector $p (t)$. The latter satisfies the
time-dependent analogue of Eq.~\eqref{eq:def-r}, which can be solved
for the entries of $r$ by M\"{o}bius inversion, giving
\begin{equation}\label{eq:p-k}
  p^{\pa}_{K} (t) \, = \sum_{\udo{J}\subseteq K} (-1)^{\lvert K \nts -
    \ts J\rvert} \ts \lambda^{\pa}_{J} (t) \, = \sum_{\udo{J}\subseteq K}
  (-1)^{\lvert K \nts - \ts J\rvert} \ts \ee^{\ts t \ts \mu^{\pa}_{\nts
      J}} \, = \sum_{\udo{J} \subseteq K} (-1)^{\lvert K \nts - \ts
    J\rvert} \exp \Bigl( t \sum_{\udo{I}\subseteq J} r^{\pa}_{I} \Bigr) .
\end{equation}
Now, we simply obtain
$M^{\pa}_{IJ} (t) = \sum_{I\cup\udo{K} = J}\ts p^{\pa}_{K} (t)$ for each
$t\geqslant 0$ as in \eqref{eq:def-prop-PM}, and thus a simple,
explicit approach to the semigroup.  \exend
\end{remark}

When, in Theorem~\ref{thm:log-relation}, we hit a situation where more
than one generator exists for $M$ (which can only happen for the
non-generic case of multiple eigenvalues, and even then at most for
very small values of $\det (M)$; compare \cite{Davies, BS3} and
references therein), only one can have Property CG. This is so
because $M = \ee^R = \ee^{R'}$ implies
$\one = \ee^R \ee^{-R'} = \ee^{R-R'}$ (because $R$ and $R'$ commute)
and thus $R=R'$ (because $R-R'$ is diagonalisable with all
eigenvalues being $0$).  In a case with more than one embedding,
two (or more) Markov semigroups cross each other in $M$, but the
extra solutions rarely have an interesting probabilistic
interpretation.

At this point, in view of Theorem~\ref{thm:log-relation}, we need to
understand when we get non-negative entries $r^{\pa}_{K}$ for all
$K\ne \varnothing$. As it turns out, this is best looked at in a
probabilistic fashion, which we do after a closer inspection of the
model structure.

\section{Intermezzo: Algebraic and analytic properties of the
  model}\label{sec:inter}

Here, we develop an alternative picture via identifying suitable
moduli spaces for our matrix classes together with some algebraic
and analytic features. Though this is equivalent to the full matrix
formulation, it directly works with the parameter vectors and
exhibits a structure of independent interest. As above, we use
$S=\{ 1, 2, \ldots , N \}$, its power set $2^S$ as a lattice with
order relation $\subseteq$, and $\RR^d$ with $d=2^{N}$ as the vector
space of row vectors, indexed by the elements of $2^S$. In particular,
we write $x\in\RR^d$ as $x = (x^{\pa}_{I})^{\pa}_{I\subseteq S}$.

Earlier, we considered the family of CM matrices $M \in \cM_d$,
which had matrix elements
$M^{\pa}_{IJ} = \sum_{I\cup \udo{K} = J} p^{\pa}_{K}$ for some $p\in\RR^d$
with non-negative entries only and $\sum_{I\subseteq S} p^{\pa}_{I} = 1$.
In other words, the $(d-1)$-dimensional probability simplex
\[
  \mathbb{S}_{d-1} \, = \, \{ p \in \RR^d : \text{all } p^{\pa}_{I}
  \geqslant 0 \, \ts \text{ and }
  \textstyle{\sum_{I\subseteq S}} \, p^{\pa}_{I} = 1 \}
\]
is the moduli space of this family of matrices. Note that the matrix
family is a closed convex subset of $\cM_d$, which is matched by
$\mathbb{S}_{d-1}$ being a closed convex simplex in $\RR^d$. It has
$d$ extremal elements, namely the vectors
$e^{(K)} = (\delta^{\pa}_{K,L})^{\pa}_{L\subseteq S}$, which was used in
Remark~\ref{rem:commute}.

In view of the role of the parameter vectors $p$ and their behaviour
under multiplication of CM matrices, we now define a multiplication
$\star : \RR^d \times \RR^d \xrightarrow{\quad} \RR^d$ by
\begin{equation}\label{eq:def-mult}
  (x \star y)^{\pa}_K \, \defeq \sum_{\udo{I} \subseteq K}
  \, \sum_{I \cup \udo{J}=K} x^{\pa}_{I} \, y^{\pa}_{J} \ts .
\end{equation}
It has a remarkable property as follows, which relates to 
Proposition~\ref{prop:diag} and Remark~\ref{rem:commute}.

\begin{lemma}\label{lem:sum-rule}
  The multiplication defined by \eqref{eq:def-mult} is commutative and
  satisfies the sum rule
\[
  \sum_{\udo{K} \subseteq S} ( x \star y)^{\pa}_{K}
  \, = \, \biggl( \, \sum_{\udo{I}\subseteq S} x^{\pa}_{I} \biggr)
  \biggl( \, \sum_{\udo{J}\subseteq S} y^{\pa}_{J} \biggr) .
\]
\end{lemma}

\begin{proof}
  By definition, one has
\[
    (x \star y)^{\pa}_{K} \, = \sum_{\udo{I}\subseteq K} 
      \sum_{I\cup\udo{J} = K} x^{\pa}_{I} \, y^{\pa}_{J} \, = \!
      \sum_{\substack{I,J\subseteq K \\ I \cup J^{\vphantom{I}} = K}}
      \! x^{\pa}_{I} \, y^{\pa}_{J} \, = \sum_{\udo{J}\subseteq K}
       \sum_{J\cup \udo{I} = K} y^{\pa}_{J} \, x^{\pa}_{I}
      \, = \, (y \star x)^{\pa}_{K} \ts ,
\]  
which establishes commutativity, while the normalisation follows from
\[
  \sum_{\udo{K} \subseteq S} (x \star y)^{\pa}_{K}
  \, = \sum_{\udo{K} \subseteq S} \, 
        \sum_{ \substack{ I,J \subseteq K \\ I\cup J^{\vphantom{I}}=K}}
          x^{\pa}_{I} \, y^{\pa}_{J} \, =  \sum_{I,J \subseteq S} 
        x^{\pa}_{I} \, y^{\pa}_{J} \, =  \sum_{\udo{I}\subseteq S} x^{\pa}_{I} \,
        \sum_{\udo{J}\subseteq S} y^{\pa}_{J} \ts ,
\]
which completes the argument.  
\end{proof}

If the entries of $x$ and $y$ sum to $\alpha$, those of $x\star y$ sum
to $\alpha^2$, whence the cases $\alpha = 1$ and $\alpha = 0$ are
special. Let us analyse this a bit more.  Clearly, if
$r,s \in \mathbb{S}_{d-1}$, then also $r\star s \in \mathbb{S}_{d-1}$,
and we get that $(\mathbb{S}_{d-1}, \star)$ is an Abelian monoid, that
is, an Abelian semigroup with unit. The latter is
$\varepsilon = (\delta^{\pa}_{\varnothing, K})^{\pa}_{K\subseteq S}$,
because $\varepsilon \star x = x$ for all $x$, as one can easily
verify.

Due to the underlying order lattice, it is possible to equip 
$\RR^d$ with a suitable norm,
\begin{equation}\label{eq:norm}
  \| x \| \, \defeq \, \max_{K\subseteq S} \; \bigl|
  \textstyle{\sum_{\udo{I}
    \subseteq K} } \; x^{\pa}_{I} \bigr| \ts ,
\end{equation}
which is the spectral radius of our original matrices
and gives the following structure.

\begin{fact}\label{fact:Banach}
  The mapping\/ $x \mapsto \| x \|$ derived from \eqref{eq:norm}
  defines a norm on\/ $\RR^d$ that is submultiplicative for\/ $\star$
  and turns\/ $(\RR^d,+,\star)$ into a Banach algebra.
\end{fact}

\begin{proof}
  $\| x \| \geqslant 0$ for all $x\in\RR^d$ is clear, while
  non-degeneracy is a consequence of the M\"obius inversion formula
  from \eqref{eq:Mobius-1}. Likewise,
  $\| \alpha \ts x \| = \lvert \alpha \rvert \, \| x \|$ holds for all
  $\alpha \in\RR$, while the triangle inequality follows from a simple
  calculation.

  The norm is submultiplicative,
  $\| x \star y \| \leqslant \| x \| \, \| y \|$ for all
  $x,y\in\RR^d$, as follows from a computation analogous to the one
  for Lemma~\ref{lem:sum-rule}. As $\RR^d$ is closed with respect to
  $\|.\|$, the Banach algebra property is clear.
\end{proof}

Next, we look at the CG matrices, hence those with a
parameter vector $r = (r^{\pa}_{I})^{\pa}_{I\subset S} \in \RR^d$
subject to the condition that $\sum_{\udo{I}\subseteq S} r^{\pa}_{I} = 0$.  
These matrices have zero row sums, but need not be
rate matrices, which they are if and only if $r^{\pa}_{I} \geqslant 0$
holds for all $I\ne\varnothing$. Here, the CG matrices form a
subalgebra of $\AAA^{(d)}_{\ts 0}$ of dimension $d-1$, and its 
moduli space is
\[
  \XX \, = \{ x \in \RR^d : x^{\pa}_{1} + \ldots + x^{\pa}_{d} = 0 \}
  \, \simeq \, \RR^{d-1} ,
\]
where $x+y$ and $x \star y$ correspond to the sum and the product of
the matrices parametrised by $x$ and $y$, respectively, with the
correct distributive law. In other words, $(\XX, +, \star)$ is an
Abelian algebra, but without a multiplicative unit. It is closed under
limits in the usual topology. \smallskip

Now, consider the Cauchy (or initial value) problem
\begin{equation}\label{eq:Cauchy}
    \dot{M} (t) \, = \, M(t) \ts Q(t) \quad \text{with} \quad
    M(0) \, = \, \one \ts ,
\end{equation}
where $\{ Q(t) : t\geqslant 0\}$ is any continuous family of CG
matrices.  Since $Q(t)$ and $Q(s)$ commute for all $t,s\geqslant 0$ by
Corollary~\ref{coro:diag}, the solution is simply given by
\begin{equation}\label{eq:sol}
     M (t) \, = \, \exp \bigl( R (t)\bigr) \quad \text{with} \quad
     R(t) \, =   \int_{0}^{t} Q (\tau) \dd \tau \ts ,
\end{equation}
as follows from the standard theory of matrix-valued, ordinary
differential equations; we refer to \cite[Ch.~IV]{Walter} for
background.  Clearly, since the algebra of CG matrices is closed under
taking limits, $R(t)$ is a CG matrix for every $t\geqslant 0$, and we
have the following result.

\begin{coro}\label{coro:flow}
   The forward flow of the Cauchy problem \eqref{eq:Cauchy},
   with\/ $\{ Q(t) : t \geqslant 0 \}$ a continuous family of\/ 
   \textnormal{CG} matrices, consists of Markov matrices with 
   Property\/ \textnormal{CM} only. Further, 
   every individual\/ $M(t)$ from this flow is embeddable into a
   time-homogeneous Markov semigroup generated by a rate matrix 
   with Property\/ \textnormal{CG}.   \qed 
\end{coro}

The continuity condition can be relaxed considerably, by replacing the
differential equation with the corresponding Volterra integral
equation; see \cite{BS4} and references therein for more. Instead, let
us interpret \eqref{eq:sol} in terms of the moduli spaces introduced
above. To do so, we write CG matrices as $Q_q$, hence indexed with the
corresponding parameter vector, and analogously use $M_p$ for CM
matrices. Now, in Eq.~\eqref{eq:sol}, we clearly have
$R(t) = R_{r(t)}$ with
\begin{equation}\label{eq:int}
    r(t) \, = \int_{0}^{t} q(\tau) \dd \tau \ts ,
\end{equation}
where $q(\tau)$ denotes the parameter vector of $Q(\tau)$. 

It is now possible to calculate $\exp \bigl( R(t) \bigr)$, for every
fixed $t\geqslant 0$, as follows, where we drop the explicit notation
for time dependence for clarity. Consider the exponential series for
$\ee^R$ and observe that $R^n$, with $n\in\NN$, has parameter vector
$r^{(n)} = r\star \cdots \star r \eqdef r^{\star n}$. Further,
$\one = R^{\ts 0}$ is a CM matrix with our unit $\varepsilon$ from
above as parameter vector. If we set $r^{\star 0}\defeq \varepsilon$,
we can now define the mapping $\nexp : \XX \xrightarrow{\quad} \RR^d$
by
\[
  r \, \mapsto \,   \nexp (r) \, \defeq
  \sum_{n=0} \myfrac{r^{\star n}}{n \ts !} \ts ,
\]
which is convergent and defines a vector whose entries sum to $1$.
More generally, $\nexp (x)$ is well defined for all $x\in\RR^d$, where
convergence follows from a Weierstrass $M$-test via the estimate
$\| \nexp (x) \| \leqslant \ee^{\|x\|}$ with the norm from
\eqref{eq:norm} and Fact~\ref{fact:Banach}.

In fact, also Euler's relation still holds for all $x\in\RR^d$, namely
\[
  \nexp (x) \, = \lim_{n\to\infty} \Bigl( \varepsilon +
  \myfrac{x}{n}\Bigr)^{\star n} ,
\]
which can then be used (as in the proof of Lemma~\ref{lem:easy-2}) to
show that $\nexp (r)$, for $r\in\XX$, is indeed a probability vector,
as it must.

\begin{coro}\label{coro:exp}
  In the setting of Corollary~$\ref{coro:flow}$, each matrix\/ $M(t)$
  from the flow is a\/ \textnormal{CM} matrix with parameter vector\/
  $p (t) = \nexp\bigl( r(t)\bigr)$ with\/ $r(t)$ as in \eqref{eq:int}.
  \qed
\end{coro}

Clearly, one has $\nexp (0) = \varepsilon$. Observing that the
binomial formula for $\star$ holds in the form
\[
     (x+y)^{\star n} \, = \sum_{m=0}^{n} \binom{n}{m} 
      \, x^{\star m} \star y^{\star (n-m)} ,
\]
we see that $\nexp (.)$ satisfies
$\nexp (x+y) = \nexp (x) \star \nexp (y)$.  Recalling the computations
from Remark~\ref{rem:simpler}, one can evaluate $\nexp (r)$ explicitly
via
\[
     \bigl( \nexp (r)\bigr)_{K} \, = \sum_{\udo{J} \subseteq K}  
     (-1)^{\lvert K-J\rvert} \exp \Bigl( \, \sum_{\udo{I}\subseteq J} 
     r^{\pa}_{I} \Bigr) \, = \sum_{\udo{J}\subseteq K}
     (-1)^{\lvert K-J\rvert} \prod_{\udo{I}\subseteq J} \ee^{r^{\pa}_{I}} .
\]
If we set
$\YY \defeq \{ y \in \RR^d : \sum_{K\subseteq S} y^{\pa}_{K} = 1 \}$, we
know that $\nexp (\XX) \subset \YY$. It would be of interest to
characterise $\nexp (\XX)$ more explicitly. Further, with
\[
    \XX_{\geqslant} \, \defeq \, \{ x \in \XX : x^{\pa}_{K} \geqslant 0
     \text{ for all } \varnothing \ne K \subseteq S \} \ts ,
\]
we clearly have
$\nexp (\XX_{\geqslant}) \subset \mathbb{S}_{d-1} \subset \YY$, where
$\nexp (\XX_{\geqslant})$ is the moduli space of the embeddable cases
in $\mathbb{S}_{d-1}$, which we still need to analyse.  So far, we
have the following.

\begin{coro}
  Let\/ $M$ be a\/ \textnormal{CM} matrix with parameter vector\/ $p$.
  Then, there is a\/ \textnormal{CG} matrix\/ $R$ with\/ $M=\ee^R$ if
  and only if the equation\/ $p = \nexp (x)$ has a solution\/
  $x\in\XX$. Further, $R$ is a Markov generator if and only if\/
  $x\in\XX_{\geqslant}$.  \qed
\end{coro}

In view of Fact~\ref{fact:Banach}, it is possible to define a
logarithm via
\[
  \nlog (\varepsilon + x) \, \defeq \sum_{n=1}^{\infty}
  \myfrac{(-1)^{n-1}}{n}\, x^{\star n} ,
\]
which is convergent for all $x\in\RR^d$ with $\| x \| < 1$. This
indeed defines the inverse function for $\nexp$, and one can now
formulate the entire embedding problem in terms of the two moduli
spaces and their behaviour under the mappings $\nexp$ and $\nlog$.
Still, the final task to understand the positivity condition needs
further insight, which turns out to have a probabilistic root. Let us
thus return to the original embeddability question.

\section{Positivity conditions and embeddability}\label{sec:pos}

Recall that $Z\subseteq S$ denotes the (random) set of objects that
are sampled in one step, with $\PP (Z=K) = p^{\pa}_{K}$ for
$K\subseteq S$.  In what follows, a crucial role will be played by the
probability of a relevant `non-event',
\begin{equation}\label{eq:non-event}
   q^{\pa}_{K} \, \defeq \, \PP ( Z \text{ avoids } K ) 
   \, = \, \PP (Z \cap K = \varnothing) \, = 
   \sum_{\udo{I} \subseteq \, \overline{\! K}} p^{\pa}_{I} \, = \, 
   1 - \! \! \sum_{\udo{J} \cap K \ne \varnothing} \! p^{\pa}_{J} \ts ,
\end{equation}
where the last step is a consequence of $p$ being a probability
vector.  In particular, one has
$q^{\pa}_{\varnothing} = \sum_{\udo{I} \subseteq S} p^{\pa}_{I} = 1$, as it
must be, and $q^{\pa}_{K} \geqslant q^{\pa}_{H}$ whenever $K\subseteq H$.
The $q^{\pa}_{K}$ are related with the eigenvalues of $M$ via
\begin{equation}\label{eq:prop-eigenvalues}
  \lambda^{\pa}_{K} \, = \, M^{\pa}_{KK} \, =  \!
  \sum_{K\cup \udo{I}=K} \! p^{\pa}_{I} \, =
  \sum_{\udo{I} \subseteq K} p^{\pa}_{I} \, = \,
  \PP (Z \cap \,\overline{\! K} = \varnothing) \, = \,
  q^{\pa}_{\, \overline{\! K}} \ts .
\end{equation}
This also gives
$\lambda^{\pa}_{\varnothing} = q^{\pa}_{S} = p^{\pa}_{\varnothing}$, which must
be positive for embeddability by Corollary~\ref{coro:no-go}. We
therefore assume, for the rest of this section, that
$p^{\pa}_{\varnothing}>0$, which then also implies that $q^{\pa}_{K}>0$
holds for all $K\subseteq S$ via \eqref{eq:non-event}.

Since we now have $\mu^{\pa}_{\, \overline{\! K}} = \log \bigl(
\lambda^{\pa}_{\, \overline{\!  K}} \bigr) = \log (q^{\pa}_{K})$, we
can use \eqref{eq:magic} to also get
\begin{equation}\label{eq:para-prob}
  r^{\pa}_{K} \, = \sum_{\varnothing \ne \udo{H} \supseteq \, \overline{\! K}}
  (-1)^{\lvert H - \, \overline{\! K} \ts \rvert } \ts \log( q^{\pa}_{H}) \, = \,
  \log \! \prod_{\varnothing \ne \udo{H} \supseteq \, \overline{\! K}} \!
  q^{(-1)^{\lvert H - \, \overline{\! K} \ts \rvert }}_{H}
\end{equation}
for $K\ne \varnothing$, where (for the embedding problem) we
need to know when they all satisfy $r^{\pa}_{K} \geqslant
0$. Alternatively, due to our new version \eqref{eq:new-version} with
the full parameter vector $r$, we can also write
\begin{equation}\label{eq:para-variant}
    r^{\pa}_{K} \, = \sum_{\udo{H}\subseteq K} 
    (-1)^{\lvert K - H \rvert} \log (q^{\pa}_{\, \overline{\! H}}) \, = \, \log 
    \prod_{\udo{H}\subseteq K} q^{(-1)^{\lvert K-H \rvert}}_{\, \overline{\! H}} ,
\end{equation}
now for all $K\subseteq S$, where
$r^{\pa}_{\varnothing} = \log (q^{\pa}_{S}) = \log (p^{\pa}_{\varnothing})$. A
simple calculation with the M\"{o}bius inversion from
Eq.~\eqref{eq:Mobius-1} then gives
\[
  \log (q^{\pa}_{K}) \, = \sum_{\udo{H} \subseteq \,\overline{\! K}}
  r^{\pa}_{H} \, = \, - \!\sum_{\udo{I} \cap K \ne \varnothing} \!
  r^{\pa}_{I} \; \leqslant \; 0 \ts .
\]
Before we continue, let us look at a simple special case.

\begin{example}\label{ex:zwei}
  If $S=\{ 1,2 \}$, we have
  $\cP(S) = \big\{ \varnothing, \{ 1 \}, \{ 2 \}, S \big\}$, and thus
  get 
\begin{align*}
  r^{\pa}_{S} \ts & = \,\log \myfrac{q^{\pa}_{S}\, q^{\pa}_{\varnothing}}
       {q^{\pa}_{\{ 1 \}} q^{\pa}_{\{ 2 \}} } \, = \,
       \log \myfrac{q^{\pa}_{S}}{q^{\pa}_{\{ 1 \}} q^{\pa}_{\{ 2 \}} } \, = \,
       \log \myfrac{\lambda^{\pa}_{\varnothing}}
       {\lambda^{\pa}_{\{1\}}\lambda^{\pa}_{\{2\}}} \, , \\
  r^{\pa}_{\{ 1 \}} \, & = \, \log \myfrac{q^{\pa}_{\{ 2 \}}}{q^{\pa}_{S}}
    \, = \, \log \myfrac{\lambda^{\pa}_{\{1\}}}{\lambda^{\pa}_{\varnothing}}
         \, , \quad r^{\pa}_{\{ 2 \}} \ts = \, \log
         \myfrac{q^{\pa}_{\{ 1 \}}}{q^{\pa}_{S}}
    \, = \, \log \myfrac{\lambda^{\pa}_{\{2\}}}{\lambda^{\pa}_{\varnothing}}
\end{align*}
together with $r^{\pa}_{\varnothing} = \log (q^{\pa}_{S})$, so the sum of
all four coefficients vanishes, as it must. Here, due to
$q^{\pa}_{K} \geqslant q^{\pa}_{H}$ for $K\subseteq H$, the conditions
$r^{\pa}_{\{ 1 \}} \geqslant 0$ and $ r^{\pa}_{\{ 2 \}} \geqslant 0$ are
automatically satisfied, which can also be understood via the
conditional probabilities
\[
  r^{\pa}_{\{ i \} } \, = \, - \log \myfrac{q^{\pa}_{S}}{q^{\pa}_{S - \{ i \}}}
  \, = \, - \log \PP \bigl( Z \cap \{ i \} = \varnothing \mid Z
  \cap (S - \{ i \}) = \varnothing \bigr) \ts .
\]
For the remaining coefficient, we have
\begin{equation}\label{eq:two-cond}
  r^{\pa}_{S} \geqslant 0 \; \Longleftrightarrow \; \Delta q \defeq
    q^{\pa}_{S \vphantom{\{\}}} - q^{\pa}_{\{ 1 \} }
    q^{\pa}_{ \{ 2 \}} \geqslant 0 \ts ,
\end{equation}
where $\Delta q$ has the interpretation of a correlation, either
between the two non-events $\{ 1 \not\in Z \}$ and $\{ 2 \not\in Z \}$
or (equivalently) between the events $\{ 1 \in Z \}$ and
$\{ 2 \in Z \}$.  Here, one has
$\PP ( i \in Z ) = p^{\pa}_{\{ i \}} \nts + p^{\pa}_{S \vphantom{\{ \}}}$
for $i\in S$ and $\PP ( Z = S) = p^{\pa}_{S}$, which then gives
\[
  \Delta q  \, = \, p^{\pa}_{\varnothing \vphantom{\{\}}} \ts
    p^{\pa}_{S \vphantom{\{\}}} - p^{\pa}_{\{ 1 \}} p^{\pa}_{\{ 2 \}}
\]
by a simple calculation via \eqref{eq:non-event} and the relation
$\sum_{K\subseteq S} p^{\pa}_{K} = 1$.   \exend
\end{example}

In Example~\ref{ex:zwei}, the required non-negativity depends on just
one condition, namely Eq.~\eqref{eq:two-cond}, which is equivalent to
\[
  \PP ( 1,2 \in Z ) - \PP (1 \in Z) \ts \PP ( 2\in Z)
  \, \geqslant \, 0 \ts .
\]
This suggests the following probabilistic interpretation. Consider the
continuous-time MCCP with $r^{\pa}_{\{ 1 \}} r^{\pa}_{\{ 2 \}} > 0$, to
avoid degeneracies, and $r^{\pa}_{\{1, 2\}} \geqslant 0$. For
$r^{\pa}_{\{1, 2\}} = 0$, we are back to Example~\ref{ex:reco} and its
independence structure, so
\[
  \PP ( 1,2 \in Z) \, = \, \bigl(1 - \ee^{- r^{\pa}_{\{1\}} t } \bigr) 
    \bigl(1 - \ee^{- r^{\pa}_{\{2\}} t } \bigr) \, = \, 
    \PP (1 \in Z) \ts \PP ( 2\in Z) \ts ,
\]
so there is no correlation. If we now add joint sampling at rate
$r^{\pa}_{\{1, 2\}} >0$, we obtain
\[
\begin{split}
  \PP (i \in Z ) \, & = \, p^{\pa}_{\{i\}} (t) + p^{\pa}_{\{1,2\}} (t) \, = \,
       1 - \ee^{- t (r^{\pa}_{\{i\}} \nts + \ts r^{\pa}_{\{1,2\}} )  } \quad 
       \text{for $i\in\{1,2\}$ and} \\
  \PP ( 1,2 \in Z) \, & = \, p^{\pa}_{\{1,2\}} (t) \, = \, 1 - 
        \ee^{- t (r^{\pa}_{\{1\}} \nts + \ts r^{\pa}_{\{1,2\}})}  -
        \ee^{- t (r^{\pa}_{\{2\}} \nts + \ts r^{\pa}_{\{1,2\}})}  +
        \ee^{- t (r^{\pa}_{\{1\}} \nts + \ts r^{\pa}_{\{2\}} \nts + 
        \ts r^{\pa}_{\{1,2\}})} ,
\end{split}
\]
by using the formula for $p^{\pa}_{K} (t)$ from \eqref{eq:p-k}. This
gives $\PP (1,2 \in Z) > \PP (1 \in Z) \ts \PP ( 2\in Z)$ and thus a
positive correlation, in line with the intuition for this simple case.
Put differently, there is no way to achieve a negative correlation
within the continuous-time process, whereas, in contrast, objects can
very well avoid each other in discrete time --- just set
$p^{\pa}_{\{ 1 \}} p^{\pa}_{\{ 2 \}} > 0$ together with
$p^{\pa}_{\{1, 2\}} = 0$.  \smallskip

These considerations might trigger the naive question whether
$r^{\pa}_{K} \geqslant 0$ for all $\varnothing \ne K \subseteq S$ could
be equivalent to $C^{\pa}_{K} \geqslant 0$, where
\begin{equation}\label{eq:corr}
     C^{\pa}_{K} \, = \sum_{\udo{\cA} \in \cP (K)}
        (-1)^{\lvert \cA \rvert - 1} \bigl( \lvert \cA \rvert - 1\bigr)!
        \prod_{A \in \cA} \PP (A \subseteq Z)
\end{equation}
is the correlation function of the events $\{ i \in Z\}$ for $i\in K$.
In fact, this is not true, but points in the right direction. To
proceed, we need to consider larger sets $S$.

Indeed, the situation gets more involved for sets $S$ with
$\lvert S \rvert \geqslant 3$.  Here, one still has
\[
  r^{\pa}_{\{ i \}} \, = \, \log \myfrac{q^{\pa}_{S-\{ i \}}}{q^{\pa}_{S}} \, = \,
  \log \myfrac{p^{\pa}_{\varnothing \vphantom{ \{\} }} + 
  p^{\pa}_{\{ i \} } }{p^{\pa}_{\varnothing \vphantom{ \{  \} }}} \, \geqslant \, 0
\]
for $i\in S$, with non-negativity for the same reason as in
Example~\ref{ex:zwei}, and the analogous interpretation in terms of
conditional probabilities. For $i, j\in S$ with $i\ne j$, we get 
\[
   r^{\pa}_{ \{ i, j  \} } \, = \, 
   \log \myfrac{q^{\pa}_{S \vphantom{\{\}} } \, q^{\pa}_{S-\{ i, j \} }}
   {q^{\pa}_{S - \{ i \} } \, q^{\pa}_{S - \{ j \} } }  \, = \, \log
   \myfrac{p^{\pa}_{\varnothing \vphantom{\{\}}} 
     \bigl( p^{\pa}_{\varnothing \vphantom{\{\}}} + p^{\pa}_{ \{ i \} }
     + p^{\pa}_{ \{ j \} } + p^{\pa}_{ \{ i, j \} } \bigr) }
   {\bigl( p^{\pa}_{\varnothing \vphantom{\{\}}} + p^{\pa}_{\{i\}}\bigr)
       \bigl( p^{\pa}_{\varnothing \vphantom{\{\}}} + p^{\pa}_{\{j\}} \bigr)} 
\]
from \eqref{eq:para-variant}, which implies
\[
  r^{\pa}_{ \{ i, j \} } \geqslant 0 \;\; \Longleftrightarrow \;\;
  q^{\pa}_{S \vphantom{\{\}}} \, q^{\pa}_{S - \{ i, j \}} \geqslant q^{\pa}_{S
    - \{ i \} } \, q^{\pa}_{S - \{ j \} } \;\; \Longleftrightarrow \;\;
  p^{\pa}_{\varnothing \vphantom{\{\}}} \, p^{\pa}_{\{ i, j \}} \geqslant
  p^{\pa}_{ \{ i \} } \, p^{\pa}_{ \{ j \} } \ts .
\]
For $S = \{ 1, 2\}$, this simplifies as stated earlier because one
then has $q^{\pa}_{S - \{ i, j \} } = q^{\pa}_{\varnothing \vphantom{\{\}}} = 1$.

Next, for $i,j,k \in S $ distinct, one obtains
\[
\begin{split}
  r^{\pa}_{ \{ i, j, k \} } \, & = \, \log \myfrac{q^{\pa}_{ S - \{ i, j, k
      \}} \, q^{\pa}_{S - \{ i \} } \, q^{\pa}_{S - \{ j \} } \, q^{\pa}_{S -
      \{ k \} } } {q^{\pa}_{S \vphantom{\{\}}} \, q^{\pa}_{S- \{ i, j \}} \,
    q^{\pa}_{S - \{ i, k \} } \, q^{\pa}_{S - \{ j, k \}}} \\[1mm]
  & = \, \log \myfrac{q^{\pa}_{S - \{ i, j, k\}}}{q^{\pa}_{S \vphantom{\{\}}}}
        - \log \myfrac {q^{2}_{S - \{ i, j, k \} } }
           {q^{\pa}_{S- \{ i, j \} } \, q^{\pa}_{S - \{ k \} }}
        - \log \myfrac{q^{2}_{S - \{ i, j, k \} } }
           {q^{\pa}_{S- \{ i, k \} } \, q^{\pa}_{S - \{ j \} }} 
     \\[1mm] & \qquad  - \log \myfrac{q^{2}_{S - \{ i, j, k \} } }
           {q^{\pa}_{S- \{ j, k \} } \, q^{\pa}_{S - \{ i \} }}
        + 2 \log  \myfrac{q^{3}_{S - \{ i, j, k \} } }
          {q^{\pa}_{S - \{ i, j \} } \, q^{\pa}_{S - \{ i, k \} } \,
           q^{\pa}_{S - \{ j, k \} }} \ts .
\end{split}
\]
Again, for $S = \{ 1, 2, 3 \}$, one gets a slight simplification due
to $q^{\pa}_{S - \{ i, j, k \} } = q^{\pa}_{\varnothing \vphantom{\{\}}} = 1$.
The quantities $r^{\pa}_{ \{ i, j, k \} }$ cannot reasonably be
expressed in terms of the parameters $p^{\pa}_{K}$.

To see what is going on in general, we define, for $A,B \subseteq S$,
the quantities
\begin{equation}\label{eq:qab}
    q^{B}_{A} \, \defeq \, \PP (Z\cap A = \varnothing \mid
       Z \cap B = \varnothing ) \, = \, 
       \myfrac{\PP (Z\cap (A\cup B) = \varnothing) }
           {\PP (Z\cap B = \varnothing)} \, = \,
       \myfrac{q^{\pa}_{A\cup B}}{q^{\pa}_{B}} \ts ,
\end{equation}
where $q^{\varnothing}_{A} = q^{\pa}_{A}$ because $q^{\pa}_{\varnothing}=1$,
and one also has 
\[
   q^{\,\overline{\! B}}_{\nts A\cup \: \overline{\! B}} \, = \,
   \myfrac{q^{\pa}_{\nts A\cup \: \overline{\! B}}}{q^{\pa}_{\ts\overline{\! B}} }
   \, = \, q^{\,\overline{\! B}}_{\nts A\vphantom{\overline{B}}} \ts .
\]
Now, we can formulate the following result.

\begin{prop}\label{prop:set-to-part}
   For any non-empty subset\/ $K\subseteq S$, we have the identity
\[
  r^{\pa}_{K} \, = \, (-1)^{\lvert K \rvert} \! \sum_{\udo{\cA} \in \cP
    (K)} (-1)^{\lvert \cA \rvert - 1} \bigl( \lvert \cA \rvert -
  1\bigr) ! \, \log \bigl( q^{\, \overline{\! K}}_{\cA} \ts \bigr) ,
\]   
where\/ $q^{B}_{\cA} \defeq \prod_{A\in \cA} q^{B}_{A}$ for any
partition\/ $\cA \in \cP (K)$ and any\/ $B\subseteq S$.
\end{prop}

\begin{proof}
From \eqref{eq:para-variant}, upon substituting $A=K\nts -H$, one gets
\[
\begin{split}
  r^{\pa}_{K} \, & = \sum_{\udo{A} \subseteq K} (-1)^{\lvert A \rvert}
  \log (q^{\pa}_{A\cup \, \overline{\! K}}) \, = \sum_{\udo{A} \subseteq
    K} (-1)^{\lvert A \rvert} \bigl( \log (q^{\pa}_{A\cup \, \overline{\!
      K}})
  - \log (q^{\pa}_{\, \overline{\!K}}) \bigr) \\[1mm]
  & = \sum_{\udo{A} \subseteq K} (-1)^{\lvert A \rvert} \log \bigl(
  q^{\, \overline{\! K}}_{A} \bigr) \, = 
  \sum_{\varnothing \ne \udo{A} \subseteq K} (-1)^{\lvert A \rvert} 
  \log \bigl( q^{\, \overline{\! K}}_{A} \bigr) \\[1mm]  
  & = \,  (-1)^{\lvert K \rvert} \!
  \sum_{\udo{\cA}\in \cP (K)} (-1)^{\lvert \cA \rvert - 1} \bigl(
  \lvert \cA \rvert - 1 \bigr) ! \, \log \bigl( q^{\, \overline{\!
      K}}_{\cA} \ts \bigr) ,
\end{split}
\]
where the second step is true because
$\sum_{\udo{A} \subseteq K} (-1)^{\lvert A \rvert} = 0$, while the
next two steps follow from \eqref{eq:qab} together with 
$q^{\overline{K}}_{\varnothing} = 1$.  Now, we can invoke
Lemma~\ref{lem:set-to-part} with $f(A) = q^{\, \overline{\! K}}_{A}$.
\end{proof}

If we compare the formula for the $r^{\pa}_{K}$ with the expression for
the $C^{\pa}_{K}$ from \eqref{eq:corr}, one notices the same basic
structure, but three important differences. First, we have the
appearance of logarithms, which is perhaps not surprising. Second, the
probabilities $\PP (A\subseteq Z)$ are replaced by the conditional
probabilities
$\PP ( Z\cap A = \varnothing \mid Z \cap \,\overline{\! K} =
\varnothing )$. Finally, there is an extra factor
$(-1)^{\lvert K \rvert}$, which is due to the appearance of non-events
instead of events.

At this point, we can wrap up the embedding structure as follows.

\begin{theorem}\label{thm:final}
  Let\/ $M$ be the Markov matrix of a discrete-time MCCP for\/
  $S=\{ 1, 2, \ldots , N\}$, so a\/ \textnormal{CM} matrix with 
  parameter vector\/ $p = (p^{\pa}_{K})^{\pa}_{K\subseteq S}$. If\/ $M$ is
  non-singular, which is equivalent with\/ $p^{\pa}_{\varnothing} > 0$,
  it has a real logarithm, $R$, in the form of the principal
  matrix logarithm.
  
  This\/ $R$ is a triangular matrix with zero row sums and real
  spectrum. It satisfies Property\/ \textnormal{CG} according to
  Definition~$\ref{def:prop-PG}$, with the parameter vector\/
  $r = (r^{\pa}_{K})^{\pa}_{K\subseteq S}$ from
  Proposition~$\ref{prop:set-to-part}$. Further, $R$ is a Markov
  generator if and only if\/ $r^{\pa}_{K} \geqslant 0$ holds for all\/
  $\varnothing\ne K \subseteq S$, in which case\/ $M$ is embeddable
  into a continuous-time\/ \textnormal{MCCP}.
  
  Finally, $R$ is the only real matrix logarithm of\/ $M$ with
  Property\/ \textnormal{CG}. In the generic case that the spectrum
  of\/ $M$ is simple, no other real logarithm of\/ $M$ exists, and\/
  $M$ is embeddable if and only if\/ $R$ is a Markov generator.
\end{theorem}

\begin{proof}
  The first part follows from Corollary~\ref{coro:no-go} and
  Theorem~\ref{thm:log-relation}, and the second from the same theorem
  in conjunction with Eq.~\eqref{eq:para-variant} and
  Proposition~\ref{prop:set-to-part}.
 
  If $R'$ is any real logarithm of $M$ in triangular form, it also has
  real spectrum because its eigenvalues are the diagonal
  elements. When it also has Property CG, its unique parameter vector
  $r'$ follows from the spectrum by Eq.~\eqref{eq:para-variant}, and
  must then agree with $r$, hence $R'=R$.  The final claim is then
  clear.
\end{proof}

If $M$ has a degenerate spectrum and a sufficiently small determinant,
other real logarithms of $M$ are possible, but never with Property
CG. In this sense, Theorem~\ref{thm:final} gives a complete answer to
the embedding problem, both in the generic case of simple spectrum and
in general, then under the constraint that an embedding should be into
a model of the same type (meaning MCCP in our case). One can say more
also in the case of a degenerate spectrum, by comparing the
centralisers of $M$ and any real logarithm of it (see \cite{Culver} as
well as \cite{BS2} and references therein for some tools). In fact, an 
embedding of $M$ always implies CG embedding, by a standard 
approximation argument fom \cite{Davies}. We leave further details to 
the interested reader.

%

\section*{Acknowledgements}

It is our pleasure to thank Jeremy Sumner for interesting discussions
on the embedding problem and for various helpful comments on the
manuscript.  We also thank an anonymous referee for 
thoughtful suggestions, which helped us to improve the
presentation. This work was supported by the German Research
Foundation (DFG), within the CRC 1283 \mbox{(Project ID 317210226)} 
at Bielefeld University.

\end{document}